\documentclass{amsart}
\usepackage{url}

\usepackage{xcolor}

\usepackage{amsmath,amsfonts,amsthm,amssymb}
\newtheorem{theorem}{Theorem}

\newtheorem{lemma}[theorem]{Lemma}

\newtheorem{remark}[theorem]{Remark}
\numberwithin{equation}{section} \numberwithin{theorem}{section}
\newcommand{\R}{\mathbb{R}}
\newcommand{\N}{\mathbb{N}}

\newcommand{\Z}{\mathbb{Z}}

\newcommand{\loc}{\mathrm{loc}}

\newcommand{\ep}{\varepsilon} 
\newcommand{\ph}{\varphi}

\DeclareMathOperator{\Div}{div}

\DeclareMathOperator{\supp}{supp}

\begin{document}
\title[Biharmonic wave maps into spheres]{Biharmonic wave maps into spheres}

\author[S.~Herr]{Sebastian Herr}
\address[S.~Herr]{Fakult\"at f\"ur
  Mathematik\\ Universit\"at Bielefeld\\ Postfach 10 01 31\\ 33501
  Bielefeld\\ Germany}
\email{herr@math.uni-bielefeld.de}

\author[T.~Lamm]{Tobias Lamm}
\address[T.~Lamm]{Department of Mathematics\\ 
Karlsruhe Institute of Technology \\ 
76128 Karlsruhe\\ Germany}
\email{tobias.lamm@kit.edu}

\author[R.~Schnaubelt]{Roland Schnaubelt}
\address[R.~Schnaubelt]{Department of Mathematics\\ 
Karlsruhe Institute of Technology \\ 
76128 Karlsruhe\\ Germany}
\email{schnaubelt@kit.edu}

\thanks{TL and RS gratefully acknowledge financial support by the Deutsche Forschungsgemeinschaft (DFG)
through CRC 1173.\\ The authors thank the referee for pointing out an error in an earlier version of the paper.}


\date{\today}


\begin{abstract} 
A global weak solution of the biharmonic wave map
equation in the energy space for spherical targets is constructed.
The equation is reformulated as a conservation law and solved by a suitable Ginzburg-Landau type approximation.
\end{abstract}

\maketitle
\section{Introduction}\label{sec:intro}
We study biharmonic wave maps $u:I\times \R^n\to S^l$, where $S^l$ is the $l$-dimensional unit sphere in 
$\R^{l+1}$, and $I\subset \R$ is an open interval.  These maps are critical points of the action functional
\[
\Phi(u):=\frac12 \int_{I \times \R^n} (|\partial_t u|^2-|\Delta u|^2)\, d(t,x)
\]
acting on functions with values in $S^l$. Here $\Delta u =(\Delta u^ i)_i$ is the extrinsic Laplacian; 
i.e., the Laplacian w.r.t.\ $x$ when considering $u$ as a map into $\R^{l+1}$. In  our main Theorem~\ref{thm:global}
we construct a global weak solution for all data in the energy space.

We introduce two equivalent versions of the biharmonic wave map system  for regular solutions.
Sufficiently smooth critical points  $u:I\times \R^n\to S^l$ of $\Phi$ satisfy
\begin{equation}\label{EQsphere0}
(\partial_t^2+\Delta^2 )u \perp T_uS^l
\end{equation}
which can be viewed as the geometric version of biharmonic wave map equation
To show this claim, for $\varphi\in C_c^\infty(I\times \R^{n},\R^{l+1})$ we consider the variation 
$u_\tau=\pi(u+\tau\varphi)$, where $\pi:\R^{l+1}\setminus \{0\} \to S^l $ denotes the retraction $\pi(y)=\frac{y}{|y|}$, 
and $\tau>0$ is small enough. We compute
\begin{align*}
\frac{\mathrm{d}}{\mathrm{d}\tau}\Phi(u_\tau)\Big|_{\tau=0}=&\int_{I \times \R^n}\langle \partial_t u ,\partial_t D\pi(u)\varphi\rangle -\langle \Delta u , \Delta  D\pi(u)\varphi\rangle d(t,x)\\
=&-\int_{I \times \R^n}\langle \partial_t^2 u +\Delta^2 u,D\pi(u)\varphi\rangle d(t,x)
\end{align*}
Choose a smooth orthonormal frame $\{v_1(t,x),\ldots,v_l(t,x)\}$ for $T_{u(t,x)}S^l$, a scalar function 
$\eta\in C_c^\infty(I\times \R^{n},\R)$ and define $\varphi_j=v_j \eta$, for $j=1,\ldots,l$. Since $D\pi(u)\varphi_j=\varphi_j$, 
for a critical point $u$ of $\Phi$ we obtain
\[
0=\int_{I \times \R^n}\langle \partial_t^2 u +\Delta^2 u,v_j\rangle \eta d(t,x).
\]
We conclude  that $\langle \partial_t^2 u +\Delta^2 u,v_j\rangle=0$ for any $j=1,\ldots,l$, which shows \eqref{EQsphere0}.

For smooth $u:I\times \R^n\to S^l$, equation \eqref{EQsphere0} is equivalent to the  PDE-version
\begin{align}
 (\partial_t^2+\Delta^2 )u=\big[\,|\Delta u|^2-|\partial_t u|^2-\Delta |\nabla u|^2
     -2\Div \langle \Delta u,\nabla u\rangle \big]\, u\label{EQsphere}
 \end{align}
of the biharmonic wave map system.
 Here and below, for any  $A\in \R^{(l+1)\times (l+1)}$ the expression 
$\Div \langle \Delta u,A\nabla u\rangle$ is shorthand
for $\sum_{i=1}^n \partial_i \langle \Delta u, A\partial_i u\rangle$, where
$\langle\cdot,\cdot\rangle$ is the scalar product in $\R^{l+1}$. We also write
$\langle \nabla \Delta u, \nabla u\rangle= \sum_{i=1}^d  \langle \partial_i  \Delta u,\partial_i u\rangle$
etc.. Moreover, $|\cdot|$ denotes the Euclidean norm in $\R^{l+1}$ and in $\R^{(l+1)(l+1)}$.

We show the above mentioned equivalence. Equation \eqref{EQsphere0} means that there is 
a function $\lambda_u: I\times \R^n\to\R$ such that $(\partial_t^2+\Delta^2 )u=\lambda_u u$.
A solution to \eqref{EQsphere} of course satisfies this identity with $\lambda_u=[\ldots]$. To see 
the converse, we multiply $(\partial_t^2+\Delta^2 )u=\lambda_u u$ by $u$ and use the product rule. 
It follows that
\begin{align*} 
\lambda_u &= \langle \partial_t^2 u, u\rangle + \langle \Delta^2 u, u\rangle
  = \partial_t \langle \partial_t u, u\rangle - \langle \partial_t u, \partial_t u\rangle 
      + \Div \langle \nabla \Delta u, u\rangle -\langle \nabla \Delta u, \nabla  u\rangle.
\end{align*}
By $|u|^2=1$, we have $2\langle \partial_k u, u\rangle=\partial_k|u|^2=0$ for $k\in\{t,1,\ldots,n\}$.
We then compute
\begin{align*} 
\lambda_u &= -|\partial_t u|^2 + \Delta \langle \Delta u, u\rangle - 2\Div \langle \Delta u,\nabla u\rangle
              + \langle \Delta u, \Delta u\rangle\\
          &=   -|\partial_t u|^2  + \Delta \Div \langle \nabla u, u\rangle 
            - \Delta \langle \nabla u, \nabla u\rangle - 2\Div \langle \Delta u,\nabla u\rangle
              + |\Delta u|^2\\
          & =|\Delta u|^2-|\partial_t u|^2-\Delta |\nabla u|^2
     -2\Div \langle \Delta u,\nabla u\rangle,
\end{align*}
as asserted.

The energy corresponding to $\Phi$ is given by 
\[
E[u](t)=\frac12 \int_{\{t\} \times \R^n} (|\partial_t u|^2+|\Delta u|^2)\,  dx.
\]
We thus introduce the space for (global) weak solutions of our problems as
\[ Z= \{ u\in L^ \infty(\R\times \R^n, \R^{l+1})\,|\, u(t,x)\in S^l \text{  a.e., } 
\partial_t u, \Delta u\in  L^2_{\mathrm{loc}}(\R;L^2(\R^n,\R^{l+1}))\}.\]
As above, one observes that 
\[\langle \Delta u, u\rangle = \Div \langle \nabla u, u\rangle - \langle \nabla u, \nabla u\rangle 
=- |\nabla u|^2,\]
so that each $u\in Z$ satisfies
\begin{equation}\label{L4}
 |\nabla u|^2\le |\Delta u| \qquad\text{and}\qquad  |\nabla u|\in L^4_{\mathrm{loc}}(\R; L^4(\R^n)).
\end{equation}
A  \emph{weak solution} of \eqref{EQsphere0} is defined as a map $u\in Z$ 
fulfilling 
\begin{equation}\label{weak0}
0 =\int_{\R\times \R^n} \big(-\langle \partial_t u, \partial_t v\rangle
+ \langle \Delta u, \Delta v\rangle \big)\, dt\, dx
\end{equation}
for all  functions $v$ belonging to
\begin{align*}
&V:=\{v\in L^\infty(\R\times \R^n,  \R^{l+1})\,|\, |\partial_t v|, |\nabla v|^ 2, |\Delta v|  
       \in L^2(\R\times \R^ n); \; \supp_t v \text{ is compact}, \\
&\qquad\qquad \qquad\qquad \text{and $v$ satisfies} \quad  v(t,x)\in T_{u(t,x)}S^l \text{ \ for a.e.  }   (t,x)\in \R\times \R^{n} \},
\end{align*}
where $\supp_t$ denotes the support of $v:\R\to L^\infty(\R^n,  \R^{l+1})$, $t\mapsto v(t,\cdot)$.
Moreover, $u\in Z$ is a \emph{weak solution} of \eqref{EQsphere} if 
\begin{equation}\label{weak}
\begin{split}
&\int_{\R\times \R^n} (-\langle \partial_t u, \partial_t \phi\rangle 
  + \langle \Delta u, \Delta \phi\rangle)\, dt\, dx\\
  & = \int_{\R\times \R^n} \big( (|\Delta u|^2-|\partial_t u|^2) \langle u,\phi\rangle 
     -  |\nabla u|^2 \,\Delta \langle u,\phi\rangle
    +2\langle \Delta u,\nabla u\rangle \, \nabla \langle u,\phi\rangle\big) dt\, dx
  \end{split}  
\end{equation}  
for all $\phi\in W$, where we put
 \[
 W:=\{v\in L^\infty(\R\times \R^n,  \R^{l+1})\,|\, |\partial_t v|, |\nabla v|^ 2, |\Delta v|  
       \in L^2(\R\times \R^ n); \; \supp_t v \text{ is compact} \}.
 \] 
 Note that the terms on the right hand side in this definition are integrable by \eqref{L4}. 
 In Lemma~\ref{divergence} we prove the equivalence of the weak solvability of \eqref{EQsphere0}
 and of \eqref{EQsphere}.

The fourth order system \eqref{EQsphere} is analogous to the (second order) wave maps system, 
see e.g.\ \cite{ShSt}. In this situation global weak solutions in the energy space have been constructed by Shatah \cite{shatah} for spherical targets and by Freire \cite{Fr} for target manifolds being homogeneous spaces. These constructions use a suitable Ginzburg-Landau type approximation of \eqref{EQsphere}. Our main result is a variant of the result of Shatah for biharmonic wave maps. 
\begin{theorem}\label{thm:global}
Let $(u_0,u_1)\in L^\infty(\R^n,\R^{l+1})\times L^2(\R^n,\R^{l+1})$ satisfy $\Delta u_0\in L^2(\R^n,\R^{l+1})$
as well as $u_0(x)\in S^l$ and $u_1(x)\in T_{u_0(x)}S^l$ for a.e.\ $x\in \R^n$. Then there is a global
weak solution $u\in Z$ of \eqref{EQsphere}  with $u(0)=u_0$ and $\partial_t u(0)=u_1$. 
Moreover, the maps $\Delta u, \partial_t u:\R\to L^2(\R^n,\R^{l+1})$ are weakly continuous and 
bounded, we have $u(t,\cdot)\in S^ l$ and  $\partial_t u(t,\cdot)\in T_{u(t,\cdot)}S^ l$ a.e.\ for each $t\in\R$,
and the difference $u-u_0$ is weakly continuous as a map from $\R$ to $H^\theta(\R^n,\R^{l+1})$ for all $\theta\in [0,2)$.
Finally, for all $t\in\R$ the solution satisfies the energy inequality
\[ E[u](t)=\frac12 \int_{\{t\} \times \R^n} (|\partial_t u|^2+|\Delta u|^2)\,  dx 
\leq E[u](0)=\frac12 \int_{\R^n} (|u_1|^2+|\Delta u_0|^2)\,  dx.  \]
\end{theorem}
Scalar fourth order wave equations, such as the beam equation, have been studied previously in \cite{HePa} or \cite{P2010}. 
In the case of biharmonic wave maps, the authors together with T. Schmid, recently showed in \cite{HLSS} a local well-posedness results for maps taking values in arbitrary compact target manifolds, assuming that the initial data are regular enough.
This result was then used by T. Schmid \cite{Schmid} in order to show the existence of a unique global smooth solution for smooth and compactly supported initial data in the cases $n=1,2$. This extends earlier work of Fan and Ozawa \cite{fan}.
Finally, we want to mention that weak solutions for the parabolic variant of the problem, the so called biharmonic map heat flow, have been constructed under certain restrictions on the dimension $n$ in \cite{Gastel}, \cite{Lamm}, \cite{Moser} and \cite{Wang}.
 
We note that there is a second functional which also deserves to be called the action functional corresponding 
to biharmonic wave maps, namely
\[
\Psi(u):=\frac12 \int_{\R\times \R^n} (|\partial_t u|^2-|(\Delta u)^T|^2)\, dt\, dx
  =\frac12 \int_{\R\times \R^n} (|\partial_t u|^2-|\Delta u|^2+|\nabla u|^4)\, dt\, dx,
\]
where $(\Delta u)^T=\Delta u+u|\nabla u|^2$ is the tangential component of the Laplacian. In this case critical 
points satisfy the PDE
\[
\partial_t^2 u+\Delta^2 u+2\Div(|\nabla u|^2 \nabla u)\perp T_uS^l
\]
or equivalently
\begin{equation}\label{EQsphere2}
\begin{split}
 \partial_t^2u+\Delta^2 u+2\Div(|\nabla u|^2 \nabla u)=u\big(
&|\Delta u|^2-|\partial_t u|^2-\Delta |\nabla u|^2\\
&-2\Div \langle \Delta u,\nabla u\rangle -2|\nabla u|^4\big).
\end{split}
\end{equation}
Due to the additional nonlinear term, our proof of Theorem
\ref{thm:global} does not extend to this
equation. 

\section{The conservation law} 

As a first result we show that the systems \eqref{EQsphere0} and
\eqref{EQsphere} are also equivalent in 
the weak sense and that they can be can be written in divergence form \eqref{eqdivergence}. The latter 
fact will be crucial for our global existence result. 
\begin{lemma}\label{divergence}
For $u\in Z$ the following assertions are equivalent.
\begin{enumerate}
\item The map $u$ is a weak solution of \eqref{EQsphere0}.
\item The map $u$ is a weak solution of \eqref{EQsphere}.
\item For all skew-symmetric matrices $\Lambda \in \R^{(l+1)\times (l+1)}$ the map $u$ is a weak solution of the system
  \begin{align}
   0=\partial_t\langle \partial_t u , \Lambda u\rangle+\Delta \langle \Delta u, \Lambda u\rangle 
   -2\Div \langle \Delta u ,\Lambda \nabla u\rangle \label{eqdivergence}
   \end{align}	
  on $\R\times \R^n$ with test functions in $W_s$, where
  \[
  W_s:=\{v\in L^\infty(\R\times \R^n)\,|\, |\partial_t v|, |\nabla v|^ 2, |\Delta v|  
       \in L^2(\R\times \R^ n); \; \supp_t v \text{ is compact} \}.
       \]
\end{enumerate}
\end{lemma}
\begin{proof}
{\it (1)$\,\Rightarrow$(3)}. Let $u$ be a weak solution of \eqref{EQsphere0}. Take
$\varphi \in V_s$ and $\Lambda \in \R^{(l+1)\times (l+1)}$ with  $\Lambda^T=-\Lambda$. 
The function $v=\varphi \Lambda u$ belongs to $V$ by \eqref{L4} and takes values
in $T_u S^l$ since $\Lambda$ is skew-symmetric. We  thus obtain
\begin{align*}
0&= \int_{\R\times \R^n} \left(-\langle \partial_t u,\partial_t (\varphi \Lambda u)\rangle 
     +\langle \Delta u,\Delta (\varphi (\Lambda u)\rangle \right)\,dt\, dx\\
&= \int_{\R\times \R^{n}} \left( -\partial_t \varphi \langle \partial_t u,\Lambda u\rangle 
     +\Delta \varphi \langle \Delta u,\Lambda u\rangle 
     +2 \nabla \varphi \langle \Delta u,\Lambda \nabla u\rangle \right)\,dt\, dx,
\end{align*}
using that $\langle \partial_t u, \Lambda \partial_t u\rangle=0= \langle \Delta u, \Lambda \Delta u\rangle$. 
Hence, $u$ is a weak solution of \eqref{eqdivergence}.

{\it (3)$\,\Rightarrow$(2)}. Let  $u$ be a weak solution of \eqref{eqdivergence}.
We employ for $1\le i<j\le l+1$ and $\omega \in S^l$ the tangent vectorfields
\[
\Lambda_{ij} \omega=(e_i\otimes e_j -e_j \otimes e_i)\omega =\omega^ie_j-\omega^je_i\in T_\omega S^l.
\]
These vectorfields span $T_\omega S^l$ since each $\xi \in T_\omega S^l$ has the representation
\[
\xi=(\omega \otimes \xi- \xi \otimes \omega)\omega 
   =\sum_{1\le i<j\le l+1} (\omega^i \xi^j-\omega^j\xi^i)\Lambda_{ij} \omega.
\]
For a given function $\phi\in W$ we deduce
\begin{align}\label{eq:Lambda}
\phi &= \langle \phi,u\rangle u+\sum_{1\le i<j\le l+1} \varphi_{ij} \Lambda_{ij} u,\qquad \text{where}\\
\varphi_{ij} &:= u^i(\phi^j-\langle \phi,u\rangle u^j)-u^j(\phi^i-\langle \phi,u\rangle u^i).	\notag
\end{align}
Note that all maps $\varphi_{ij}\Lambda_{ij}u$ and $\langle \phi,u\rangle u$ belong to $W$, and  
$\varphi_{ij}$ to $W_s$. Assertion (3) then yields
\begin{align*}
&\int_{\R\times \R^n} \big( -\langle \partial_t u,\partial_t (\varphi_{ij} \Lambda_{ij}u)\rangle 
  +\langle \Delta u ,\Delta (\varphi_{ij} \Lambda_{ij}u\rangle \Big)\,dt\, dx\\
&= \int_{\R\times \R^{n}} \Big( -\partial_t \varphi_{ij} \langle \partial_t u,\Lambda_{ij} u\rangle 
  +\Delta \varphi_{ij} \langle \Delta u,\Lambda_{ij} u\rangle 
 +2\nabla \varphi_{ij} \langle \Delta u,\Lambda_{ij} \nabla u\rangle \Big) \,dt\, dx\\
&=0,
\end{align*}
where two terms vanish because of the skew-symmetry of $\Lambda_{ij}$.
For the normal component, we compute
\begin{align*}
& \int_{\R\times \R^n} \left( -\langle \partial_t u,\partial_t (\langle \phi,u\rangle u)\rangle 
 +\langle \Delta u ,\Delta (\langle \phi,u\rangle u)\rangle \right)\,dt\, dx\\
&= \int_{\R\times \R^n} \!\big({-}|\partial_t u|^2\langle u,\phi\rangle +|\Delta u|^2 \langle u,\phi \rangle 
  +2\langle \Delta u,\nabla u  \rangle \,\nabla \langle \phi,u\rangle
   +\langle \Delta u,u\rangle \,\Delta \langle \phi,u\rangle \big)\,dt\, dx\\
    &= \int_{\R\times \R^n} \big( (|\Delta u|^2-|\partial_t u|^2) \langle u,\phi\rangle 
     -  |\nabla u|^2 \,\Delta \langle u,\phi\rangle
    +2\langle \Delta u,\nabla u\rangle \, \nabla \langle u,\phi\rangle\big) dt\, dx
\end{align*}  
since $\langle \partial_t u, u\rangle= 0$ and $\langle \Delta u,u\rangle=-|\nabla u|^2$. Summing up, the 
decomposition \eqref{eq:Lambda} implies  that $u$ solves \eqref{EQsphere} weakly.

{\it (2)$\,\Rightarrow$(1)}. Let $u$ be a weak solution of \eqref{EQsphere}. For test functions $v\in W$  
taking values in $T_{u(t,x)}S^l$ equation \eqref{weak0} follows from \eqref{weak} since then $\langle u,v\rangle =0$. 
\end{proof}

Note that the conservation law can also be obtained via Noether's theorem. For any map $u$ the action 
functional $\Phi$ is invariant under rotations $R(\tau)u=\exp(\tau\Lambda)u$. This fact implies 
\begin{align*}
0&= \frac{\mathrm{d}}{\mathrm{d}\tau} \Phi(R(\tau)u,\Omega)\Big|_{\tau=0}
=\int_\Omega \left( \langle \partial_t u, \partial_t(\Lambda u)\rangle 
    -\langle \Delta u,\Delta (\Lambda u)\rangle \right)\,dt\, dx\\
&= \int_\Omega \left( \partial_t \langle \partial_t u,\Lambda u\rangle 
 +\Delta \langle \Delta u,\Lambda u \rangle -2\Div \langle \Delta u, \Lambda \nabla u\rangle \right)\,dt\, dx\\
&\quad -\int_\Omega \langle (\partial_t^2+\Delta^2 )u,\Lambda u\rangle \,dt\, dx. 
\end{align*}
for each subset $\Omega\subset \R\times \R^n$. The second integral vanishes since 
$(\partial_t^2+\Delta^2)u\perp T_uS^l$, and hence we have derived again the conservation law 
\eqref{eqdivergence}.

\begin{remark}\label{rem:tang-lapl}
Similar to Lemma \ref{divergence} one can prove that
a (smooth) map $u$ is a solution of \eqref{EQsphere2} iff for all $\Lambda \in \R^{(l+1)\times (l+1)}$ with 
$\Lambda^T=-\Lambda$ we have
\begin{align}
0=\partial_t\langle \partial_t u , \Lambda u\rangle+\Delta \langle \Delta u, \Lambda u\rangle 
   -2\Div \langle \Delta u ,\Lambda \nabla u\rangle+2\Div \langle |\nabla u|^2 \nabla u,\Lambda u\rangle . 
   \label{eqdivergence2}
\end{align}
\end{remark}

\section{Existence of a global weak solution}\label{subsec:gws}
In this section we construct a global weak solution of \eqref{EQsphere} using 
a penalization method as in \cite{Fr}. To this end, we fix an increasing function
$\chi\in C^\infty([0,\infty))$ with $\chi(s)= 1$ for all $s\ge \frac12$ and $\chi(s) =s$ for all 
$s\le \frac14$.  We then define the smooth map $F:\R^{l+1} \to\R$ by
\[
F(x)=\chi\circ (|x|^2-1)^2.
\]
Observe that   $F$ is bounded,  its derivatives are compactly  supported, $F^{-1}(0)=S^l$, and $\nabla F(x)=0$
if $|x|=1$. 

For $\ep>0$ and initial functions
 $(u_0,u_1)\in L^\infty(\R^n,\R^{l+1})\times L^2(\R^n,\R^{l+1})$ with $\Delta u_0\in L^2(\R^n,\R^{l+1})$
 and $u_0(x)\in S^l$ for almost every $x\in \R^n$, we look at the 
 auxiliary system
 \begin{equation}\label{eqapprox}
\begin{split}
\partial_t^2u_\ep+\Delta^2 u_\ep+\tfrac1\ep\nabla F(u_\ep)&=0\\
u_\ep(0,\cdot)=u_0,\ \ \ \partial_t u_\ep(0,\cdot)&=u_1, 
\end{split}
\end{equation}
without requiring that $u(t,x)\in S^l$ if $t\neq0$ a.e..
We point out that the initial value $u_0$ here (and below) is not square-integrable, which causes
technical difficulties. In contrast to the wave map case in \cite{Fr},
solutions of \eqref{eqapprox} do not possess finite speed of propagation so that standard cut-off arguments cannot be used.
 Instead we look for (distributional) solutions of the form $u_\ep=u_0+v_\ep$ for a function 
$v_\ep$ solving the shifted system
\begin{equation}\label{eqapprox1}
\begin{split}
\partial_t^2v_\ep+\Delta^2 v_\ep+\tfrac1\ep\nabla F(u_0+v_\ep) +\Delta^2 u_0 &=0  \\
v_\ep(0,\cdot)=0,\ \ \ \partial_t v_\ep(0,\cdot)&=u_1, 
\end{split}
\end{equation}
weakly, with test  functions in $W$. 
  For brevity, we sometimes write $H^k$ instad of $H^k(\R^n,\R^{l+1})$ 
for $k\in \Z$, and analogously for other function spaces.

We use  the following fact. Let $Z\hookrightarrow Y$ be reflexive Banach spaces, $Z$ be dense 
in $Y$, and $f:\R\to Y$ be a weakly continuous function which is essentially bounded with values in $Z$.
Then $f$ is bounded and weakly continuous as a map into $Z$.

\begin{lemma}\label{lem:v-ep}
Let $\ep>0$ and  $(u_0,u_1)$ belong to $L^\infty(\R^n,\R^{l+1})\times L^2(\R^n,\R^{l+1})$ with 
$\Delta u_0\in L^2(\R^n,\R^{l+1})$ and $u_0(x)\in S^l$ for almost every $x\in \R^n$. Then  there is a distributional
 solution $u_\ep=u_0+v_\ep$ of \eqref{eqapprox} such that 
 $\Delta v_\ep, \partial_t v_\ep \in L^\infty(\R, L^2(\R^n,\R^{l+1}))$,
the functions $v_\ep:\R\to H^2(\R^n,\R^{l+1})$ and $\partial_t v_\ep :\R\to L^2(\R^n,\R^{l+1})$
 are weakly continuous, and $v_\ep$ solves \eqref{eqapprox1} weakly with test functions in $W$. 
 For all $t\in\R$, we have the energy inequality
\begin{align}
E_\varepsilon [u_\varepsilon](t)
 &:= \int_{\{t\}\times \R^n} \left( \frac12\, |\partial_t u_\varepsilon |^2+\frac12\, |\Delta u_\varepsilon |^2 
  +\frac{1}{\varepsilon} F(u_\varepsilon )\right)\, dx \nonumber \\
&\le E_\varepsilon[u_\varepsilon](0)= \frac12 \int_{\R^n} \left( |u_1|^2+|\Delta u_0|^2\right)\,dx =:E_0.\label{energycons}
\end{align}
\end{lemma}
\begin{proof}
1) To construct the function $v=u_\ep-u_0$,
we first study a regularized problem (and we drop the subscript $\ep$).
Let $u_0$ and $u_1$ be the given data. By means of standard mollifiers,  we obtain functions 
$u_{1,i}$ in $H^2$ converging to $u_1$ in $L^2$ as $i\to\infty$, as well as 
 $u_{0,j}\in L^\infty$ such that $\Delta u_{0,j}$ and $\Delta^2 u_{0,j}$ belong to  $L^2$,
the maps  $u_{0,j}$ tend to $u_0$ pointwise a.e.\ and with a uniform bound, and $(\Delta u_{0,j})$
converges to $\Delta u_{0}$ in $L^2$ as $j\to\infty$. Finally, let $\alpha_k$ be the characteristic 
function of the ball $B(0,k)$ in $\R^n$. We now introduce the modified equation
\begin{equation}\label{eq:v-ep}
\begin{split}
\partial_t^2v+\Delta^2 v+\tfrac1\ep\alpha_k\nabla F(u_{0,j}+v)+ \Delta^2 u_{0,j}&=0,   \\
v(0,\cdot)=0, \ \ \ \partial_t v(0,\cdot)&=u_{1,i}.
\end{split}
\end{equation}

Define $X=  H^2(\R^n,\R^{l+1}) \times L^2(\R^n,\R^{l+1})$. We have
$(v(t), \partial_tv(t))\in X$ and we look at the operator matrix 
\[A=\begin{pmatrix} 0 & -I \\ \Delta^2 & 0\end{pmatrix}, \qquad 
 D(A)= H^4(\R^n,\R^{l+1}) \times H^2(\R^n,\R^{l+1})\subset X.\]
Using the group version of the Lumer-Phillips theorem, see Corollary~II.3.6 of \cite{EN}, one checks
that $-A$ generates a strongly continuous (unbounded) group. Moreover, the map
\[ G: X\to X, \; G(\ph,\psi)=  \begin{pmatrix} 0 \\  \tfrac1\ep\alpha_k\nabla F(u_{0,j}+\ph )+ \Delta^2 u_{0,j} \end{pmatrix}\]
is globally Lipschitz and $C^1$. (For the differentiability one can employ the Sobolev embedding
$H^ 2\hookrightarrow L^p$ for some $p>2$.).
Slight variants of Theorems~6.1.2 and 6.1.5 in \cite{Pa} hence
provide a unique global solution $v=v_{i,j,k} \in  C(\R,H^4) \cap C^1(\R, H^2) \cap C^2(\R,L^2) $
of the system \eqref{eq:v-ep} in this case.  We can now differentiate the energy 
\[\tilde{E}_{j,k}[v](t):=\int_{\{t\}\times \R^n} 
  \left( \frac12\, |\partial_t v |^2+\frac12\, |\Delta (u_{0,j}+v) |^2 
  +\frac{1}{\varepsilon} \alpha_k F(u_{0,j}+v)\right)\, dx \]
 with respect to $t\in\R$.  Integration by parts yields
 \begin{align*} 
 \partial_t \tilde{E}_{j,k}[v](t) 
  &=\! \int_{\{t\}\times \R^n}\!\!\big( \langle \partial_{t} v,\partial_{tt} v\rangle 
             + \langle \Delta^2 (u_{0,j} \!+v),\partial_{t} v\rangle
             + \langle \tfrac1\ep\alpha_k\nabla F (u_{0,j}\!
    +v),\partial_{t} v\rangle\big) dx\\
    & =0,\\
\tilde{E}_{j,k}[v](t)&=\int_{\R^n} \left( \frac12\, |u_{1,i} |^2+\frac12\, |\Delta u_{0,j} |^2 
               +\frac{1}{\varepsilon} \alpha_k F(u_{0,j})\right)\, dx    
\end{align*}
for all $t\in\R$. In the next steps, we perform the limits $i,j,k\to\infty$ one after the other.
We will not relabel subsequences.

\smallskip

2) As in Theorem~6.1.2 of \cite{Pa}, the solution $(v,\partial_t v)$ of \eqref{eq:v-ep} depends continuously 
in $X$ on the initial data. The sequence  $(v_{i,j,k})_i$ thus tends in $C(\R,H^2) \cap C^1(\R, L^2)$ to a function
$v_{j,k}$ for all $j,k\in\N$, and a subsequence also converges
pointwise a.e.\ in $(t,x)$. Note that, here and throughout the paper,
the space $C(\R,X)$ (for a normed space $X$) consists of  all continuous functions $v:\R \to
X$, and convergence therein refers to locally uniform convergence
(similarly for $C^k(\R,X)$). Consequently, the map 
$v_{j,k}$ satisfies the initial conditions $v_{j,k}(0,\cdot)=0$ and  $\partial_t v_{j,k}(0,\cdot)=u_1$, and it
solves the PDE in \eqref{eq:v-ep} weakly with test functions in $W$. 

For a fixed $t\in\R$,  a further subsequence $(v_{i,j,k}(t,\cdot))_i$ tends to $v_{j,k}(t,\cdot)$ a.e..
Hence, the above energy equality leads to the identity
\[\tilde{E}_{j,k}[v_{j,k}](t)= \int_{\R^n} \left( \frac12\, |u_{1} |^2+\frac12\, |\Delta u_{0,j} |^2 
               +\frac{1}{\varepsilon} \alpha_k F(u_{0,j})\right)\, dx=: \tilde{E}^{j,k}_0\]
for all $j,k\in\N$ and $t\in\R$. 

\smallskip

3) Now, we pass to the limit $j\to\infty$ for each fixed $k\in\N$. 
 Because of the cut-off $\alpha_k$ and $F(u_0)=0$, the energies $\tilde{E}^{j,k}_0$ tend to
\[E_0:= \frac12\int_{\R^n} \big(  |u_{1} |^2+ |\Delta u_{0} |^2 \big)\, dx,\]
so that $\tilde{E}_{j,k}[v_{j,k}](t)$ is dominated by a number $c_k$ for all $j\in\N$ and $t\in\R$.
This estimate leads to the convergence
\begin{align*}
\Delta v_{j,k} \rightharpoonup \ph_k \ \ \ \text{weak}^*  \text{ in }
 L^\infty(\R;L^2)\quad \text{ and } \quad
\partial_t v_{j,k} \rightharpoonup \psi_k \ \ \ \text{weak}^*  \text{ in } L^\infty(\R;L^2)
\end{align*}
as $j\to\infty$. The functions $\ph_k$ and  $\psi_k$  inherit the energy bound by $E_0$.  We further obtain 
the estimate
\begin{align*}
\|v_{j,k}(t)\|_{L^2} = \Big\|\int_0^t \partial_s v_{j,k}(s) \,ds\Big\|_{L^2} \le 2c_k^{1/2}\,|t|\le 2m c_k^{1/2} 
\end{align*}
for all $t\in [-m,m]$ and $j,k\in \N$. The sequence $(v_{j,k})_j$ is  thus bounded in the spaces $L^\infty(J;H^2)$ 
and $W^{1,\infty}(J;L^2)$ for each $k\in\N$ and each bounded interval $J\subseteq \R$.  Proposition~1.1.4 in 
\cite{Lu} implies the interpolative embedding 
\begin{equation}\label{eq:int}
L^\infty(J;H^2)\cap W^{1,\infty}(J;L^2)\hookrightarrow C^{1-\beta}(J;H^{2\beta})
\end{equation}
for $\beta \in (0,1)$. So, by the Arzel{\`a}-Ascoli theorem, $(v_{j,k})_j$ tends to a function $v_k$ strongly
in  $C(J;H^\theta_{\loc})$ for each $\theta\in[0,2)$ and hence in $C(\R ;H^\theta_{\loc})$
and pointwise a.e., for a 
diagonal sequence. A standard test function argument then yields that $\ph_k=\Delta v_k$ and $\psi_k=\partial_t v_k$.
In particular, $v_k$ belongs to $L^\infty(J;H^2)\cap W^{1,\infty}(J;L^2)\cap C(\R ;H^\theta)$, $v_k(0)=0$, and $\Delta v_k$ is weakly 
continuous with values in $L^2$. Moreover, it satisfies the energy inequality
\[\tilde{E}_{k}[v_k](t):=\int_{\{t\}\times \R^n} 
  \left( \frac12\, |\partial_t v_k |^2+\frac12\, |\Delta (u_{0}+v_k) |^2 
  +\frac{1}{\varepsilon} \alpha_k F(u_{0}+v_k)\right)\, dx\le E_0\]
 for all $t\in\R$ and $k\in\N$.

Since the nonlinear term has compact support in space, we next deduce that $v_k$ satisfies the PDE 
in \eqref{eq:v-ep} for $u_0$ instead of $u_{0,j}$ weakly with test functions in $W$. This equation further 
shows that the weak derivative $\partial^2_{t}v_k$ actually belongs to $L^\infty_{\textrm{loc}}(\R; H^{-2})$
so that $\partial_t v_k$ is continuous from $\R$ to $H^{-2}$ and, as seen above,  essentially bounded 
in $L^2$. As a result, the map $t\mapsto\partial_t v_k(t)$ is bounded and weakly continuous  in $L^2$. 
Since $\partial_t v_{j,k}$ converges weak$^*$  in $L^\infty(\R;L^2)$ and vanishes at $t=0$, we conclude 
that  $\partial_t v_k(0)=0$.
 
 \smallskip

4) In a final step, we let $k\to\infty$. We can proceed as in Step 3) to construct a limit
function $v$ with the desired properties. There is  only one  difference in the derivation of
the PDE for $v$. To apply the
dominated convergence theorem, observe that $|\nabla F(u_0+v_k)|$ is bounded by $c\,|v_k|$ and that 
a converging sequence in $L^2$ has a subsequence with a majorant in $L^2$. Finally, the function
$u_\ep=u_0+v$ satisfies the assertions.
\end{proof}

Based on the energy estimate \eqref{energycons}, we can now pass to the limit $\ep\to0$ in  \eqref{eqapprox}.
The special form of the penalization term implies that the resulting weak limit $u$ takes values in $S^l$.
As in \cite{Fr}, we employ the equation \eqref{eqdivergence} in divergence form to show that $u$  indeed
solves of \eqref{EQsphere} weakly. To identify its initial values, we have to assume that $u_1$ maps
into the tangent space of $S^l$.

\begin{proof}[Proof of Theorem \ref{thm:global}]
1) We use the functions $u_j=u_0+v_j$ from Lemma~\ref{lem:v-ep}, where $v_j=v_{\ep_j}$ 
for some $\ep_j\to0^+$. Let $\Lambda \in \R^{(l+1)\times (l+1)}$ be 
skew-symmetric and $\ph\in C_c^\infty(\R\times \R^n)$. We take $\ph\Lambda u_j$ as a test function
for \eqref{eqapprox}. (It does not belong to $V$, in general, but the regularity provided by  Lemma~\ref{lem:v-ep}
suffices here.) 
Since $\nabla F(u_j)$ is a scalar multiple of  $u_j$, we can argue
as in the first part of  the proof of Lemma~\ref{divergence} and conclude that $u_j$ fulfills
the equation
\begin{align}
0=\partial_t\langle \partial_t u_j , \Lambda u_j \rangle
  +\Delta \langle \Delta u_j , \Lambda u_j \rangle 
  -2\Div \langle \Delta u_j ,\Lambda \nabla u_j \rangle.\label{conservationapprox}
\end{align}
in the distributional sense.

2) Starting from the energy estimate \eqref{energycons}, we can next pass to the limit $\ep_j\to0$ as in Step 3)
of the proof of Lemma~\ref{lem:v-ep} (again without relabelling subsequences). The functions 
$v_j$ then converge strongly in $C(\R; H^\theta_{\loc})$ for $\theta<2$ and pointwise a.e.\ 
to  a map $v\in C(\R; H^\theta)$. Moreover,  $\partial_t v_j$ and $\Delta v_j$ 
tend to $\partial_t v$ and $\Delta v$ weak$^*$ in $L^\infty(\R, L^2)$. Combining these facts, we infer that
$\Delta v:\R\to L^2$ is bounded and weakly continuous. 
The limit $u:=u_0+v$ thus  satisfies $u(0,\cdot)=u_0$ and $E[u](t)\le E_0$ for all $t\in\R$. 
Thanks to \eqref{conservationapprox} and the convergence  of $v_j$, the function
$u$ solves \eqref{eqdivergence} distributionally. 

The energy bound \eqref{energycons} further says that $ \| F(u_j(t,\cdot))\|_1 \le \ep_j E_0$ for all $j\in\N$
and $t\in\R$. For each bounded interval $J\subseteq \R$, Fatou's Lemma now implies that
\[
\int_{J\times \R^n} F(u)\, dx\, dt\le \liminf_{j\to \infty} \int_{J\times \R^n} F(u_j)\, dx \, dt =0.
\]
Hence, $F(u)=0$ and therefore $u(t,x)\in S^l$ for a.e.\ $(t,x)\in\R\times \R^n$. The continuity of
$u-u_0:\R\to L ^2$  then implies that  $u(t,x)$ belongs  $S^l$ for each $t\in\R$ and a.e.\ $x\in\R^ n$.
Since $u\in Z$, the map $\nabla u$ is contained  $L^\infty(\R, L^4)$ by \eqref{L4}. 
We can now deduce that $u$ weakly 
solves  \eqref{eqdivergence} with test functions in $W_s$, and so $u$ is a weak
solution of \eqref{EQsphere} by Lemma~\ref{lem:v-ep}.
Moreover, the equation $|u|^2=1$ yields $\langle \partial_t u,u\rangle=0$ so that  $\partial_t u(t,x)$ is contained
in the tangent space $T_{u(t,x)}S^ l$ for a.e.\ $(t,x)$. 

3) We still have to show the weak continuity of $\partial_t u, \Delta u:\R\to L^2$ and that $\partial_t u(0,\cdot)=u_1$. 
So far we know that the first map is essentially bounded.  Let $\Lambda \in \R^{(l+1)\times (l+1)}$ again be 
skew-symmetric. The equation \eqref{eqdivergence} and the above stated
regularity properties of $u$ imply that $\partial_t\langle \partial_t u, \Lambda u\rangle$ is bounded in 
$H^ {-2}+ W^ {-1,4/3}$. Hence, the function $\psi:t\mapsto \langle \partial_t u, \Lambda u\rangle$ is continuous 
in this space. Consequently, $\psi$ is bounded  and weakly continuous in $L^2$. Step 2) implies that for a.e.\ $t\in\R$
the vector $\partial_t u(t,x)$ belongs to  $T_{u(t,x)}S^ l$ for a.e.\ $x\in\R^ n$. 
In view of \eqref{eq:Lambda}, by modifying  $\partial_t u(t,\cdot)$ for $t$ in set of measure 0  we obtain a 
representative $\partial_t u$ which is bounded and weakly continuous as a map from $\R$ to $L^2$.

Next, we multiply the equations \eqref{conservationapprox} for $u_j$ and \eqref{eqdivergence} for $u$
 by a  function $\varphi\in C^\infty_c(\R\times \R^n)$. We integrate by parts in $t\in [0,\infty)$ 
 with values in $H^ {-2}+ W^ {-1,4/3}$ and subtract the two resulting equations, which yields
\begin{align*}
\int_{\{0\}\times \R^n} \langle u_1- \partial_t u ,\Lambda u_0\rangle \varphi \,dx
  &=  \int_0^\infty \int_{\R^n} \left( -\langle \partial_t u_j , \Lambda u_j\rangle
    +\langle \partial_t u , \Lambda u\rangle \right)\partial_t \varphi \,dx\,dt\\
  &\quad  + \int_0^\infty \int_{\R^n} \left( \langle \Delta u_j, \Lambda u_j\rangle
   -\langle \Delta u, \Lambda u\rangle \right)\Delta \varphi \,dx\,dt \\
&\quad+2 \int_0^\infty \int_{\R^n} \left( \langle \Delta u_j ,\Lambda \nabla u_j\rangle
  -\langle \Delta u ,\Lambda \nabla u\rangle\rangle \right) \nabla \varphi \,dx\,dt.
\end{align*}
By Step 2), the right hand side converges to zero as $j\to \infty$ so that 
\[
\langle \partial_t u(0) -u_1,\Lambda u_0\rangle =0.
\]
As both $u_1$ and $\partial_t u(0)$ belong to $T_{u_0}S^l$ a.e., we conclude that  $\partial_t u(0)=u_1$ and thus $\partial_t u(t,\cdot)\rightharpoonup u_1$ weakly in $L^2(\R^n)$. 
\end{proof}

\bibliographystyle{amsplain}

\end{document}